\documentclass[final,12p,times,reqno]{amsart}

\usepackage[colorlinks=true,linkcolor=black, citecolor=blue, urlcolor=blue]{hyperref}
\usepackage{amsfonts,latexsym}
\usepackage{mdwlist}  %%% suspend and resume enumerate
\usepackage{amssymb}
\usepackage{amsmath, times}

\author[ N. Attia]{Najmeddine Attia}
\address{Najmeddine Attia \\ Department of  Mathematics \\ Faculty of Sciences of Monastir \\  University of Monastir\\ Monastir-5000, Tunisia}
\email{najmeddine.attia@gmail.com}

\author[ H. Jebali ]{Hajer Jebali}
\address{Hajer Jebali\\ Department of  Mathematics \\ Faculty of Sciences of Monastir \\  University of Monastir\\ Monastir-5000, Tunisia}
\email{hajer.jebali@fsm.rnu.tn}

\author[ R. Guedri ]{Riheb Guedri}
\address{Riheb Guedri\\ Department of  Mathematics \\ Faculty of Sciences of Monastir \\  University of Monastir\\ Monastir-5000, Tunisia}
\email{rihabguedri096@gmail.com}

\newcommand{\R}{\mathbb R}

\newcommand{\N}{\mathbb N}

\newcommand{\X}{{\mathbb X}}
\newcommand{\HHH}{{\mathcal H}}
\newcommand{\W}{{\mathcal W}}

\newcommand{\F}{{\mathcal F}}

\newcommand{\HH}{{\mathsf H}}

\newcommand{\wt}{\widetilde}

\newtheorem{theorem}{Theorem}
\newtheorem*{theoremA}{Theorem A}
\newtheorem*{theoremB}{Theorem B}
\newtheorem*{theoremC}{Theorem C}

\newtheorem{lemma}{Lemma}
\newtheorem*{Proof of Theorem B}{Proof of Theorem B}
\newtheorem{proposition}{Proposition}
\newtheorem{corollary}{Corollary}

\newtheorem{definition}{Definition}
\newtheorem{remark}{Remark}

\newtheorem{example}{Example}

\usepackage{rotating}

\DeclareMathOperator{\supp}{supp} 
\DeclareMathOperator{\diam}{diam} \numberwithin{equation}{section}

\title[ On a class of  Hausdorff  measure of  cartesian product sets ]{ On a class of Hausdorff   measure of cartesian product sets in metric spaces}

\begin{document}

\maketitle

\begin{abstract}
 In this paper,  we study, in a separable  metric space, a class of Hausdorff  measures $\HHH_\mu^{q, \xi}$ defined using a measure $\mu$ and a premeasure $\xi$.  We discuss  a Hausdorff  structure of product sets. Weighted  Hausdorff  measures $\W_\mu^{q, \xi}$  appeared as an important tool when studying the  product sets. When $\mu$ and $\xi$ are blanketed, we prove that $\HHH_\mu^{q, \xi} = \W_\mu^{q, \xi}$.
As an application, the case when $\xi$ is defined as the Hausdorff function is considered.

  \bigskip


\end{abstract}

\maketitle

\section{Introduction }

 Let $(\X, \rho)$ and  $(\X', \rho')$ be two separable  metric spaces. For $\mu\in\mathcal{P}(\X)$, the family of  Borel probability measures on $\X$,  and $a>1$, we write
 $$
P_a(\mu)=\limsup_{r\searrow 0}\left( \sup_{x\in\supp\mu}
\displaystyle\frac{\mu\big(B(x, ar)\big)}{\mu\big(B(x, r)\big)}\right).
 $$
We will now say that  $\mu$ is a blanketed measure if there exists $a>1$ such that $P_a(\mu)<\infty$. It is
easily seen that the exact value of the parameter $a$ is
unimportant since $P_a(\mu)<\infty$, for some  $a>1$ if and only if
$P_a(\mu)<\infty$, for all $a>1$. Also, we will write
$\mathcal{P}_D( \X) $ for the family of blanketed Borel probability
measures on $\X$. We
can cite as classical examples of blanketed measures, the
self-similar measures and the self-conformal ones \cite{Ol1}. \\

Let $\mu\in  \mathcal{P}(\X)$, $\nu\in  \mathcal{P}(\X')$, $q, t, s\in \R$ and denote by $\HHH_\mu^{q,t}$ the multifractal Hausdorff measure introduced in \cite{Ol1}.  Then, there exists a finite positive constant $\gamma$ such that
 \begin{equation}\label{classical}
\HHH_{\mu\times \nu}^{q, t+ s}  (E\times F) \ge\gamma \HHH_\mu^{q, t} (E) \HHH_\nu^{q,s}(F),
\end{equation}
for any $E\subseteq \X$ and $F\subseteq \X'$. This result has been proved in \cite{Olsen96} for any subsets $E$ and $E'$ of Euclidean spaces $\R^d$ $(d\ge 1)$ provided that $\mu$ and $\nu$ are blanketed measures and in \cite{Attia20} by investigating some density results.  The disadvantage of density approach  includes the inability to handle sets of measure $\infty$.  Moreover, if $q=0$, then the inequality \eqref{classical}  was shown in \cite{Besicovitch45}  under certain conditions and  in \cite{Marstrand54} without any restrictions. Similar results were proved for packing measure and Hewitt-Stromberg measure (see for example \cite{Haase90a,
Haase90b, Taylor94, Howroyd96,  Attia20a, Attia20, Rihab, Attia22c, Attia23}). A generalized Hausdorff measure ${\mathcal H}^{q,t}_{\mu, \nu}$ was introduced by Cole in \cite{Cole}, where the author gave a general formalism for the multifractal analysis of a probability  measure $\mu$ with respect to another measure $\nu$, (see also \cite{Naj} for more details about ${\mathcal H}^{q,t}_{\mu, \nu}$).\\

 \medskip

 Let's denote by ${\mathcal B}(\X)$ the family of closed balls on $\X$ and $\Phi(\X)$ the class of premeasures, i.e.,  every increasing function $\xi : {\mathcal B}(\X) \to [0, +\infty]$ such that $\xi(\emptyset ) = 0$. We will denote by $\Phi_D(\X)$  the set of all premeasures $\xi\in\Phi(\X)$ satisfying  $\xi(B(x, 2 r)) \le K \xi(B(x,r))$ for some constant $K>1$ and for all $x\in \X$ and $0<r \le 1$.
  We consider a general construction, in a separable  metric space, of the   Hausdorff  measure $\HHH_\mu^{q,\xi}$ defined using a measure $\mu$ and a premeasure $\xi$ and we prove  in Subsection \ref{sec_density} that : if  $\mu$, $\nu$, $\xi$ and $\xi'$ are blanketed and if $\HHH_\mu^{q, \xi}(E)$ and $\HHH_\nu^{q, \xi'}(E')$ are finite, then
$$
\HHH_{\mu\times\nu}^{q,\xi_0}(E\times E')\ge \HHH_\mu^{q,\xi}(E) \HHH_\nu^{q, \xi'}(E')\,,
$$

where $\xi_0$ is the cartesian product measure generated from the product of $\xi$ and $\xi'$ and defined on $\mathcal{B}(\X\times\X')$ by  $\xi_0(B\times B') =\xi(B)\xi'(B')$ (see Theorem \ref{product_partial}).

 We also introduce and study a weighted generalized Hausdorff measure $\W_\mu^{q, \xi}$, for  $\xi \in \Phi(\X)$,
  (see definition in Section \ref{Prel}). Weighted and centered covers of a set $E$ are employed, in which non-negative weights are associated  with the covering sets. They appeared as an important tool when studying product sets.

    \medskip

    Throughout this paper,  we allow the cartesian product measure $\xi_0$  to be defined on $\mathcal{B}(\X)\times\mathcal{B}(\X')$   with the
    convention,  that $0 \times \infty= \infty\times 0 = 0$. Next, we give our first main results on  the weighted Hausdorff measure (Theorem
    A). Nevertheless, the equation \eqref{eq_product_W} is not true in
zero-infinite case, that is when one with zero and the other with
infinite measure. To cover this case we shall require some
restrictions. More precisely,  we suppose that, for a given real
 $q$, $(\X, \xi, \mu)$  satisfy  the following asymption \\
 \begin{eqnarray}\label{hyp1}
 & & \X \;\; \text{can be covered by countable numbers of balls} \; (B_i)_i \; \text{of arbitrarily} \notag \\
&& \text{small  diameters and}\;  \mu(B_i)^q \xi(B_i)\;  \text{ is
finite}.
\end{eqnarray}

    \begin{theoremA}
Let $q\in \R$, $\xi\in \Phi(\X)$ and $\xi'\in \Phi(\X')$. Assume that $(\X, \xi, \mu)$ and $(\X', \xi', \nu)$ both satisfy \eqref{hyp1}  then for any $E\subseteq \X$ and $F\subseteq \X'$ we have
\begin{equation}\label{eq_product_W}
{\W}_{\mu\times \nu}^{q, \xi_0}(E\times F) ={\W}_{\mu}^{q, \xi}(E)  {\W}_{\nu}^{q, \xi'}(F).
\end{equation}
\end{theoremA}

\medskip

 It is natural to ask whether the wheighted and Hausdorff measures are equal. A sufficient condition is given in the following theorem.

\begin{theoremB} \label{W=H}
Let  $\mu\in {\mathcal P}_D(\X)$, $q\in \R$ and $\xi \in \Phi_D(\X)$.  then, for any $E\subseteq \X $ we have
$$
\W_\mu^{q, \xi}(E)= \HHH_\mu^{q, \xi}(E).
$$
\end{theoremB}

Now we are able to give our main result on the Hausdorff measure.

\begin{theoremC}\label{product_WH}
Let $q\in \R$, $\xi\in \Phi(\X)$ and $\xi'\in \Phi(\X')$. Assume that $(\X, \xi, \mu)$ and $(\X', \xi', \nu)$ both satisfy \eqref{hyp1}   then
\begin{equation}\label{eq_thB}
 {\W}_{\mu}^{q, \xi}(E)  {\HHH}_{\nu}^{q, \xi'}(F)
 \le {\HHH}_{\mu\times \nu}^{q, \xi_0}(E\times F) \le  {\HHH}_{\mu}^{q, \xi}(E)  {\HHH}_{\nu}^{q, \xi'}(F).
\end{equation}
\end{theoremC}

\medskip
\noindent Let's remark that we can similarly prove that
\begin{equation*}
{\HHH}_{\mu}^{q, \xi}(E)   {\W}_{\nu}^{q, \xi'}(F)
 \le {\HHH}_{\mu\times \nu}^{q, \xi_0}(E\times F) \le  {\HHH}_{\mu}^{q, \xi}(E)  {\HHH}_{\nu}^{q, \xi'}(F).
\end{equation*}
\medskip

The proofs of our main results are  given in section \ref{proofs}.
Theorems A and C may be compared with the results  proved in
\cite{kelly73} where the author  uses, to cover a set $E$, any
subsets of $\X$ not necessarily centered in $E$.  The uses of
centered covering makes the regularity of Hausdorff and weighted
Hausdorff measures not trivial. This fact will be discussed in
Subsection \ref{regular}. Let's also mention  that results in
\cite{kelly73} didn't include the case $q<0$.
\medskip

We denote by  $\F$ the set of all Hausdorff functions, that is, the set of all  right  continuous and monotonic  increasing functions
  $h$ defined for $r\in [0, +\infty]$ with $h(0)=0$ and $h(r)>0$, for all $r>0$.
 If $\xi\in \Phi(\X)$ is defined by a Hausdorff function
 $h$, that is, $\xi(\emptyset )=0$ and, for all $B\in { \mathcal B}(\X)$,
 $$
   \xi(B) = h(\diam_\rho B), \quad \text{if}\;\; B\neq \emptyset,
 $$
 then $\xi$ will denoted by $\xi_h$. Here $\diam_\rho B$ designs the diameter of $B$ with respect to $\rho$, when there
  is no confusion
 we will simply denote  $\diam B$.  The generalized Hausdorff and weighted Hausdorff
  measures generated from a Hausdorff function $h$ will be denoted by
  $\HHH_\mu^{q, h}$ and $\W_\mu^{q,h}$ respectively.
For,  $h, h'\in {\F}$, we define the Hausdorff function $h\times h'$ by
\[h\times h'(r)=h(r)h'(r),\quad \quad \text{for all}\ r\ge 0.\,\]
As an application of Theorem C, we will prove in Section \ref{application_h} that
$$
\HHH_{\mu\times\nu}^{q,h\times h'}(E\times F)\geq \HHH_{\mu\times\nu}^{q, \xi_0}(E\times F)\ge {\W}_{\mu}^{q, h}(E)   {\HHH}_{\nu}^{q, h'}(F).
$$

%%%%%%%%%%%%%%%%%%%%%%%%%%%%%%%%%%%%%%%%%%%%%%%%%%%%%%%%%%%%%%%%%%%%%%%%%%%%%%%%%%%%

 \section{Construction of generalized measures}\label{Prel}
%%%%%%%%%%%%%%%%%%%%%%%%%%%%%%%%%%%%%%%%%%%%%%%%%%%%%%%%%%%%%%%%%%%%%%%%%%%%%%%%%%%%%ùù

 %%%%%%%%%%%%%%%%%%%%%%%%%%%%%%%%%%%%%%%%%%%%%%%%%%%%%%%%%%%%%%%%%%%%%%%%%%%%%%%%%%%%
 \subsection{Generalized Hausdorff measure}
%%%%%%%%%%%%%%%%%%%%%%%%%%%%%%%%%%%%%%%%%%%%%%%%%%%%%%%%%%%%%%%%%%%%%%%%%%%%%%%%%%%%%%%%%%%%

 Let $\delta >0$, a sequence of  closed  balls  $\{ B_i\}_i $   is called a centered $\delta-$cover of a set $E$ if, for all $i\ge 1$,  $B_i$ is centered in $E$,  $\diam B_i \le 2 \delta$ and $E\subseteq \bigcup_{i=1}^{\infty} B_i$.  Let  $\mu\in\mathcal{P}(\X)$, $ \xi \in \Phi(\X)$ and $q\in \R$, we write
\begin{equation*}
\begin{split}
{{\mathcal H}}^{q,\xi}_{\mu, \delta}(E) =&\displaystyle\inf\Big\{ \sum_i
\mu\big(B(x_i,r_i)\big)^q\xi \big( B(x_i, r_i)\big);\;\Big(B(x_i,r_i)\Big)_i\;\text{is a centered}\\
&\; \delta\text{-cover of}\; E\Big\},
  \end{split}
 \end{equation*}
 if $E\neq \emptyset$ and ${{\mathcal H}}^{q,\xi}_{\mu, \delta}(\emptyset) =0$,  with the conventions $0^q = \infty$ for $q\leq 0$ and ${{\mathcal H}}^{q,\xi}_{\mu, \delta}(E)$ is given infinite value if no centered $\delta$-cover of $E$ exists.  Now we define
 $$
 {{\mathcal H}}^{q,\xi}_{\mu,0 }(E) = \displaystyle\sup_{\delta>0} {{\mathcal H}}^{q,\xi}_{\mu, \delta}(E)$$
and
$$
{\mathcal H}^{q,\xi}_{\mu}(E)=\displaystyle\sup_{F\subseteq
E}{{\mathcal H}}^{q,\xi}_{\mu,0} (F). $$ We will prove that the
function ${\mathcal H}^{q,\xi}_{\mu}$ is a metric outer measure and
thus a measure on the Borel family of subsets of $\X$. The measure
${\mathcal H}^{q,\xi}_{\mu}$ is of course a
generalization of the centered Hausdorff measure ${\mathcal C}^t$ and multifractal Hausdorff measure  ${\mathcal H}^{q,t}_{\mu}$ \cite{Raymond88, Ol1} or ${\mathcal H}^{q,t}_{\mu,\nu}$ \cite{Naj}. %In particular, if $\xi \big( B(x, r)\big) =(2r)^t, \; (t>0)$,   it is easily seen that  one has  ${\mathcal H}^{0,t}_{\mu} = {\mathcal C}^t$ and
%$$2^{-t} {\mathcal H}^{0,t}_{\mu}\leq {\mathcal H}^{t}\leq {\mathcal
%H}^{0,t}_{\mu},$$
%where ${\mathcal H}^{t}$  denote   the $t$-dimensional Hausdorff measure.

\begin{theorem}\label{metric_H}
Let $\mu\in\mathcal{P}(\X)$, $q \in\mathbb{R}$ and $\xi \in \Phi(\X)$. Then
\begin{enumerate}
\item $\HHH_\mu^{q,\xi}$ is a metric outer measure on  $\X$ and thus measure
on the Borel family of subsets of \; $\X$;
\item for any $E\subseteq \X$, we can find a Borel set $B$ such that  $E\subseteq B$ and
$$ \HHH_{\mu,0}^{q,\xi}(B)= \HHH_{\mu,0}^{q,\xi}(E).$$
\end{enumerate}
\end{theorem}
\begin{proof}
The proof of this Theorem  is straightforward and mimics that in Theorem \ref{metric_W}.
\end{proof}

%%%%%%%%%%%%%%%%%%%%%%%%%%%%%%%%%%%%%%%%%%%%%%%%%%%%%%%%%%%%%%%%%%%%%%%%%%%%%%%%%%%%

                                \subsection{Weighted generalized Hausdorff measure}
%%%%%%%%%%%%%%%%%%%%%%%%%%%%%%%%%%%%%%%%%%%%%%%%%%%%%%%%%%%%%%%%%%%%%%%%%%%%%%%%%%%%
In the following we define the weighted generalized Hausdorff  measure.
A sequence $ (c_i , B_i)_{i\ge 1}$  of pairs, with $c_i$ a nonnegative real number and $B_i$ a closed ball of $\X$,  is said to be a weighted  cover of  $E$ if
 $$
 {\chi}_E \le \sum_{i=1}^{\infty}  c_i {\chi}_{B_i}
 $$
that is for all points $ x$ of $E$ we have
\begin{equation}\label{weight}
\sum_{i=1}^{\infty} \Big\{ c_i \;; \; x\in B_i \Big\}\ge 1.
\end{equation}
 In addition,  for $\delta >0$, we say that $ (c_i ; B_i)_{i\ge 1}$ is a weighted and centered $\delta-$cover of $E$ if
\begin{itemize}
\item  it is a weighted cover of $E$;
\item  for all $i\ge 1$,  $B_i$ is centered in $E$ and $\diam B_i \le 2 \delta$.
\end{itemize}
We denote  the family of all weighted and centered $\delta$-covers of $E$ by $ \Upsilon_\delta(E)$. Write
\begin{equation} \label{w-delta}
\begin{split}
{\mathcal{W}}_{\mu,\delta}^{q,\xi}(E)= &\inf\Big\{\sum_i c_i \mu(B_i)^q \xi(B_i); \; \;\; (c_i, B_i)_i\in \Upsilon_\delta(E) \Big\}
\end{split}
 \end{equation}
if $E\neq \emptyset$ and $ { \mathcal{W}}_{\mu,\delta}^{q,\xi}(\emptyset) = 0$, with the convention that $\W^{q,\xi}_{\mu, \delta}(E)$ is given infinite value if no weighted and centered $\delta$-cover of $E$ exists.  Now,  by applying  the standard construction we obtain
 the   weighted generalized Hausdorff  ${ \mathcal{W}}_{\mu}^{q,\xi}$ defined by
 \begin{eqnarray*}
 &&{ \mathcal{W}}_{\mu,0}^{q,\xi}(E)= \sup_{\delta >0} { \mathcal{W}}_{\mu,\delta}^{q,\xi}(E)= \lim_{\delta\to 0}{ \mathcal{W}}_{\mu,\delta}^{q,\xi}(E)\\
&&{ \mathcal{W}}_{\mu}^{q,\xi}(E)= \sup_{F\subseteq E} { \mathcal{W}}_{\mu,0}^{q,\xi}(F).
\end{eqnarray*}
\begin{remark}\label{W<H}
If the weights $c_i$ are restricted to the value unity, then \eqref{weight} requires, only, that each point of $E$ be covered once. That is covered  in the normal sense. It follows that
$$
 { \mathcal{W}}_{\mu,\delta}^{q,\xi}(E) \le   { \HHH}_{\mu,\delta}^{q,\xi}(E)\quad \text{and then }\quad  { \mathcal{W}}_{\mu}^{q,\xi}(E) \le   { \HHH}_{\mu}^{q,\xi}(E).
$$

\end{remark}

\begin{theorem}\label{metric_W}
Let $\mu\in\mathcal{P}(\X)$, $q\in\mathbb{R}$ and $\xi \in \Phi(\X)$. Then
\begin{enumerate}
\item $\W_\mu^{q,\xi}$ is a metric outer measure on  $\X$ and thus measure
on the Borel family of subsets of \; $\X$;
\item for any  $E\subseteq \X$, we can find a Borel set $B$ such that  $E\subseteq B$
$$ \W_{\mu,0}^{q,\xi}(B)= \W_{\mu,0}^{q,\xi}(E).$$
\end{enumerate}
\end{theorem}
\begin{proof}
\begin{enumerate}
\item   It is clear that $\W_\mu^{q,\xi}$ is increasing and satisfies $\W_\mu^{q,\xi}(\emptyset )= 0$. Therefore, to prove that $\W_\mu^{q,\xi}$ is an outer measure, we only have to prove that, for any sequence $\{E_n\}_{n\ge 1}$ of subsets of $\X$, we have
\begin{equation}\label{additivite}
\W_\mu^{q,\xi}\Big(\bigcup_{n=1}^{\infty} E_n \Big) \le \sum_{n=1}^\infty \W_\mu^{q,\xi}(E_n).
\end{equation}
Let $\epsilon >0$ and $\delta >0$.  For each $n\ge 1$, let $\widetilde E_n\subseteq E_n$. Then  we can  find a weighted and centered  $\delta$-cover $(c_{ni}, B_{ni})_{i\ge 1}$ of $\widetilde E_n$ such that
$$
\sum_{i=1}^{\infty} c_{ni} \mu(B_{ni})^q \xi(B_{ni}) \le \W^{q,\xi}_{\mu, \delta} (\widetilde E_n)+ \frac{\epsilon}{2^n}.
$$
Since $(c_{ni}, B_{ni})_{n, i\ge 1}$, reordered as a single sequence, forms a   weighted and centered $\delta$-cover of $\bigcup_{n=1}^{\infty} \widetilde E_n$, we have
\begin{eqnarray*}
 \W^{q,\xi}_{\mu, \delta} \Big(\bigcup_{n=1}^{\infty} \widetilde E_n \Big) &\le&   \sum_{n=1}^\infty  \sum_{i=1}^\infty c_{ni} \mu(B_{ni})^q \xi(B_{ni})\le  \sum_{n=1}^\infty  \W^{q,\xi}_{\mu,\delta} (\widetilde E_n) + \epsilon.\\
&\le&  \sum_{n=1}^\infty  \W^{q,\xi}_{\mu, 0} ( \widetilde E_n) + \epsilon \le  \sum_{n=1}^\infty  \W^{q,\xi}_{\mu} (E_n) + \epsilon .\\
 \end{eqnarray*}
By letting $\delta$ and $\epsilon$ to zero we obtain $ \W^{q,\xi}_{\mu,0} \Big(\bigcup_{n=1}^{\infty} \widetilde E_n \Big)\le  \displaystyle\sum_{n=1}^\infty  \W^{q,\xi}_{\mu} (E_n)$. Consequently \eqref{additivite} holds, so that $\W^{q,\xi}_{\mu}$ is an outer  measure. \\

Now we will prove that the measure $\W_\mu^{q,\xi}$ is metric.  For this,  let $E,F \subseteq\X$ such that $\rho(E, F) =  \inf\left\{ \rho(x,y) , x\in E, y\in F \right\} >0.$ Since $\W_\mu^{q,\xi}$ is an outer measure,  it suffices to prove that
\begin{equation}\label{metric}
\W_\mu^{q,\xi}\Big(E\bigcup F\Big)\geq\W_\mu^{q,\xi}(E)+ \W_\mu^{q,\xi}(F).
\end{equation}
We may assume that $\W_\mu^{q,\xi}\Big(E\bigcup F\Big)$  is finite. Let $E_1\subseteq E$,  $F_1\subseteq F$ and $(c_i, B_i)$ be a weighted and centered  $\delta$-cover of $E_1\cup F_1$ such that $2\delta < \rho(E, F)$. Let
  $$I=\Big\{i;\: B_i\bigcap E_1\neq\emptyset\Big\} \qquad \text{and}\qquad J=\Big\{i;\: B_i \bigcap  F_1\neq\emptyset\Big\}.$$ It is clear that $\{ c_i, B_i\}_{i\in I}$ is a weighted and centered  $\delta$-cover of  $E_1$ and $\{ c_i, B_i\}_{i\in J}$ is a weighted and centered $\delta$-cover of  $F_1$. It  follows  that
$$\W_{\mu, \delta}^{q,\xi}\Big(E_1\bigcup F_1\Big)\geq\W_{\mu, \delta}^{q,\xi}(E_1)+
\W_{\mu,\delta}^{q,\xi}(F_1)$$
and then
$$\W_\mu^{q,\xi}\Big(E\bigcup F\Big)\ge\W_{\mu,0}^{q,\xi}\Big(E_1\bigcup F_1\Big)\geq \W_{\mu, 0}^{q,\xi}(E_1)+ \W_{\mu,0}^{q,\xi}(F_1).$$
Since $E_1$ is  arbitrary subset of $E$ and $F_1$ is  arbitrary subset of $F$ we get \eqref{metric}.
\item Let $E\subseteq \X$. We may suppose that $\W_\mu^{q,\xi}$ is finite, since, if this is not the case, we have $\W_\mu^{q,\xi}(\X) = \W_\mu^{q,\xi}(E)$.  Remark that the infimum of \eqref{w-delta} remains the same when taken over strict weighted and centered $\delta$-cover of $E$,  that is,  a weighted and centered $\delta$-cover  of $E$ such that
$$
\sum_{i=1}^{\infty} \Big\{ c_{i} \;; \; x\in B_{i} \Big\} > 1,
$$
for all point $x\in E$. Now, consider, for each integer $n$, a strict weighted and centered $\delta$-cover $\{ c_{ni}, B_{ni}\}_{i\ge 1}$ of  $E$ such that
\begin{equation}\label{regular}
 \sum_{i=1}^\infty c_{ni} \mu(B_{ni})^q \xi(B_{ni}) \le \W_{\mu, \delta}^{q,\xi}(E) + \frac{1}{n}.
\end{equation}
For each $n$, consider the set
$$
B_n := \Big\{ x  :\;\;  \sum_{\stackrel{i=1}{ x \in B_{ni}}}^{\infty} c_{ni}>1 \Big\}.
$$
It is clear that $E \subseteq B_n$ and, by \eqref{regular},
\begin{equation}\label{regular1}
\W_{\mu, \delta}^{q,\xi}(B_n) \le \W_{\mu, \delta}^{q,\xi}(E) + \frac{1}{n}.
\end{equation}
In addition, for $x\in B_n$, there exists $k = k(n,x)$ such that
$$
\sum_{\stackrel{i=1}{ x \in B_{ni}}}^{k} c_{ni}>1,
$$
and  $x\in \displaystyle\bigcap_{\stackrel{i=1}{x\in B_{ni}}}^{k} B_{ni} \subseteq B_n.$ Therefore, $B_n$ can be expressed as the countable union  of  such finite intersection and thus, by taking $B=\displaystyle\bigcap_{n=1}^{\infty} B_n$ and using \eqref{regular1}, we obtain,
$$ \W_{\mu,\delta}^{q,\xi}(B)= \W_{\mu, \delta}^{q,\xi}(E).$$
Therefore, we can choose, for each integer $j$ a Borel set $B_j$ such that $E\subseteq B_j$ and $ \W_{\mu, 1/j}^{q,\xi}(B_j)= \W_{\mu, 1/j}^{q,\xi}(E)$. Let $B=\displaystyle\bigcap_{j=1}^{\infty} B_j$ we get
$$
E\subseteq B \quad \text{and}\quad \W_{\mu,0}^{q,\xi}(B)= \W_{\mu,0}^{q,\xi}(E).
$$
\end{enumerate}
\end{proof}
\begin{remark}\label{metric_W_0}
We can deduce from the proof of Theorem \ref{metric_W} that, for any sets $E, F\subseteq \X$,
$$
\W_{\mu,0}^{q,\xi} (E\cup F) \le \W_{\mu,0}^{q,\xi}(E) + \W_{\mu,0}^{q,\xi}(F)
$$
and we have the equality if  $ \rho(E, F) >0$.
\end{remark}

\begin{lemma}\label{W=H=infty}
Let  $\mu\in {\mathcal P}(\X)$, $q\in \R$ and $\xi \in \Phi_D(\X)$.  Then, for any $E\subseteq \X$ there exists a constant $\beta$ such that
\begin{equation}\label{infty}
 \beta  \,\HHH_{\mu}^{q, \xi}(E)\le W_{\mu}^{q, \xi}(E).
 \end{equation}
provided that $q\le 0$ or $q>0$ and $\mu\in {\mathcal P}_D(\X)$. In particular, we have
$$\HHH_\mu^{q, \xi}(E) =+ \infty \iff \W_\mu^{q, \xi}(E)= +\infty.
$$
\end{lemma}
\begin{proof}
Let $\delta>0$, $F\subseteq E\subseteq \X$ and  $\{ c_i, B_i\}_{i\ge 1}$ is a weighted and centered  $\delta$-covering of  $F$, where $B_i := B(x_i, r_i)$. Using Corollary 4.4 in \cite{Esmayli20}, there exists a subfamily $\{B_{i_j}\}_{j\ge 1}$ of balls such that $F\subseteq \bigcup_{j\ge 1} 3 B_{i_j}$ and
$$
\sum_{j\ge 1}  \mu\big(B(x_{i_j},r_{i_j})\big)^q \xi\big( B(x_{i_j},r_{i_j})\big) \le 8 \sum_{i\ge 1}  c_i \mu\big(B(x_{i},r_{i})\big)^q \xi\big( B(x_{i},r_{i})\big),
$$
where $3 B_{i_j} : = B(x_{i_j}, 3 r_{i_j})$. Therefore, there exists a constant $C$ such that
\begin{eqnarray*}
\HHH_{\mu, \delta}^{q, \xi}(F)&\le& \sum_{j\ge 1}  \mu\big(B(x_{i_j}, 3 r_{i_j})\big)^q \xi\big( B(x_{i_j}, 3r_{i_j})\big)\\
%&\le& \begin{cases}
%3^t \sum_{j\ge 1}  \mu\big(B(x_{i_j}, r_{i_j})\big)^q \big( 2 r_{i_j}\big)^t & q\le 0 \\
%&\\
% 3^t\,  C \sum_{j\ge 1}  \mu\big(B(x_{i_j},  r_{i_j})\big)^q \big( 2 r_{i_j}\big)^t & q>0 \;\; \text{and}\;\;  \mu\in {\mathcal P}_D(\X)  \\
%\end{cases}\\
&\le& 8\,   C \sum_{i\ge 1}  c_i \mu\big(B(x_{i},  r_{i})\big)^q  \xi \big(B(x_i, r_i)\big)
%&\le& \begin{cases}
%8 \times 3^t \sum_{i\ge 1}  c_i \mu\big(B(x_{i}, r_{i})\big)^q \big( 2 r_{i}\big)^t & q\le 0 \\
%&\\
% 8\,\times 3^t\,  C \sum_{i\ge 1}  c_i \mu\big(B(x_{i},  r_{i})\big)^q \big( 2 r_{i}\big)^t & q>0 \;\; \text{and}\;\;  \mu\in {\mathcal P}_D(\X)  \\
%\end{cases}
\end{eqnarray*}
Now, taking the infimum over all weighted and centered $\delta$-coverings $\{ c_i, B_i\}_{i\ge 1}$ of $F$ proves that
$$
\HHH_{\mu, \delta}^{q, \xi}(F)\le8\,   C\; \W_{\mu,
\delta}^{q, \xi}(F).
$$
Letting $\delta \to 0$, to get
$$
(8\,  C)^{-1} \, \HHH_{\mu,0}^{q, \xi}(F)\le\;
\W_{\mu,0}^{q, \xi}(F)\le W_{\mu}^{q, \xi}(E).
$$
Since $F$ is arbitrarily chosen, we obtain $ (8\,  C)^{-1} \,
\HHH_{\mu}^{q, \xi}(E)\le W_{\mu}^{q, \xi}(E)$.%\le \HHH_{\mu}^{q, \xi}(E)$.
%This gives the desire result when $\HHH_{\mu}^{q, \xi}(E)=\infty$.
\end{proof}

\begin{example} \label{W=H=0}
In this example we take $\X$ to be satisfying the Besicovitch
covering theorem (Theorem \ref{BCT}).   Let $\mu\in {\mathcal
P}(\X),$ and $\xi \in \Phi(\X)$. Assume that there exist  $p, a,b \ge 0$
such that $q\ge 1-p$ and an increasing continuous  function
$\varphi: [0,+\infty]\to [0, +\infty]$ with
$\displaystyle\lim_{r\to0}\varphi(r)=0$ such that
   \begin{equation}\label{muphi}
a\;   \mu(B(x,r))^p \varphi(r)  \le  \xi(B(x,r)) \le b\;  \mu(B(x,r))^p \varphi(r)
   \end{equation}
   for all $x\in \X$  and $r >0$.
  This holds,  for example, if $$\xi(B(x,r))=r^t, \qquad t>0$$
   or $$\xi(B(x,r))= \nu(B(x,r))^p h(r)$$ with $\mu \sim \nu$ and $h$ is a Hausdorff function.
   Let  $R>0$, $0<\delta <R$ and $E\subseteq B(0, R)$. By the  Besicovitch covering theorem, we can extract from
 $$
\Big \{ B(x, r), \;\; x\in E, 0< r\le \delta \Big \}
$$
 a centered $\delta$-covering $\{B(x_i, r_i)\}$ of $E$ with the overlap controlled by a constant $\gamma$. Therefore,
 \begin{eqnarray*}
 \HHH_{\mu, \delta}^{q,\xi}(E) &\le& \sum_i \mu\big(B(x_i,r_i)\big)^q \xi(B(x_i,r_i)) \le b\;  \varphi(\delta) \Big(\sum_i \mu\big(B(x_i,r_i)\big)\Big)^{q+p}\\
 &\le&  b\; \varphi(\delta)  \gamma^q \Big( \mu \big(B(0, 2R)\big)\Big)^{q+p}.
 \end{eqnarray*}
It follows, by letting $\delta$ to  $0$, that for all $E\subseteq B(0, R)$, we have  $ \HHH_{\mu}^{q,\xi}(B(0, R)) = 0$ for all $R>0$. Thus
$\W_{\mu}^{q,\xi}(\X) = \HHH_{\mu}^{q,\xi}(\X)= 0.$
Therefore, for any  $q\ge 1-p$ and $E\subseteq \X$, we have
   $$
   \W_{\mu}^{q,\xi}(E) = \HHH_{\mu}^{q,\xi}(E)=0
   $$

\end{example}

%%%%%%%%%%%%%%%%%%%%%%%%%%%%%%%%%%%%%%%%%%%%%%%%%%%%%%%%%%%%%%%%%%%%%%%%%%%%%%%%%%%%%%%%
\subsection{Regularity of the weighted generalized  Hausdorff measure}\label{regular}
%%%%%%%%%%%%%%%%%%%%%%%%%%%%%%%%%%%%%%%%%%%%%%%%%%%%%%%%%%%%%%%%%%%%%%%%%%%%%%%%%%%%

In the following, we will prove that $\W_{\mu}^{q,\xi}$  is Borel
regular measure, that is, for all $E\subseteq \X$ there exists a
Borel set $B$ such that $$
\W_{\mu}^{q,\xi}(E)=\W_{\mu}^{q,\xi}(B).$$ This is done by the
construction of  new  measure $\widetilde\W^{q,\xi}_\mu$,  in a
similar manner to the  weighted and centered generalized Hausdorff
measure $\W^{q,\xi}_\mu$   but using the class of all covering balls
in the definition rather than the class of all centered balls. We
will prove that $\widetilde\W^{q,\xi}_\mu$ is Borel regular and
$\widetilde\W^{q,\xi}_\mu$ is comparable to $ \W^{q,\xi}_\mu$ (see
\eqref{WW}).
\\

%A similar result can be proven for $\HHH_\mu^{q, \xi}$.
\begin{theorem}\label{th_regular}
Let $\mu\in\mathcal{P}_D(\X)$, $q\in\mathbb{R}$ and $\xi\in\Phi_D(\X)$.
Then $\W_{\mu}^{q,\xi}$ and $\HHH_{\mu}^{q,\xi}$ are  Borel regular.
Moreover, if $q\le 0$, then these measures are  Borel regular  even
if $\mu$ is not not blanketed.
\end{theorem}

\begin{proof}
 We will prove the result for $\W_{\mu}^{q,\xi}$. The proof can be written in a similar way for $\HHH_{\mu}^{q,\xi}$. Let $(c_i, B_i)_{i\ge 1}$ be a  weighted cover of a set $E$. For $\delta >0$, $ (c_i ; B_i)_{i\ge 1}$ is said to be  a weighted $\delta-$cover of $E$ if,   for all $i\ge 1$,  $\diam B_i \le 2 \delta$.  In this case, the center of $B_i$ does not necessarily belong to $E$.  We denote  the family of all weighted $\delta$-covers of $E$ by $ \widetilde\Upsilon_\delta(E)$. For $\mu\in\mathcal{P}(\X)$,  $q\in\mathbb{R}$, $\xi\in \Phi(\X)$,  $E \subseteq{\X}$ and $\delta>0$   we write
\begin{eqnarray}\label{ww-delta}
&&  {  \widetilde{\mathcal W}}_{\mu,\delta}^{q,\xi}(E)= \inf\left\{\sum_i c_i \mu(B_i)^q \xi(B_i); \;\; (c_i, B_i)_i\in \widetilde\Upsilon_\delta(E) \right\},
 \end{eqnarray}
 if $E\neq \emptyset$ and $\widetilde { \mathcal{W}}_{\mu,\delta}^{q,\xi}(\emptyset) = 0$.  Now, we define
 \begin{eqnarray*}
 &&\widetilde{ \mathcal{W}}_{\mu}^{q,\xi}(E)= \sup_{\delta >0} \widetilde{ \mathcal{W}}_{\mu,\delta}^{q,\xi}(E)= \lim_{\delta\to 0} \widetilde{ \mathcal{W}}_{\mu,\delta}^{q,\xi}(E)\\
\end{eqnarray*}

Now, it is not difficult to prove that, $\widetilde{\W}_\mu^{q,\xi}$
is a metric outer measure on \, $\X$ and thus a measure on the Borel
family of subsets of \; $\X$.  In addition,  $\widetilde
\W_{\mu}^{q,\xi}$ is Borel  regular. The proof is  straightforward
and mimics that in Theorem \ref{metric_W}. Since  there exists a positive constant
$\alpha$ such that, for all set $E\subseteq \X$,
\begin{equation}\label{WW}
\widetilde{\W}_\mu^{q,\xi} (E)  \le {\W}_{\mu,0}^{q,\xi}(E)\;\;  \text{and} \; \; {\W}_\mu^{q,\xi}(E) \le \alpha \, \widetilde{\W}_\mu^{q,\xi}(E),
\end{equation}
 then ${\W}_\mu^{q,\xi}$ is Borel regular (see Lemma 7 and
Theorem 5 \cite{Sch} for the key ideas  of the proof).  Moreover, if
$q\le 0$, then  \eqref{WW} holds even for $\mu$ not blanketed.

\end{proof}

%%%%%%%%%%%%%%%%%%%%%%%%%%%%%%%%%%%%%%%%%%%%%%%%%%%%%%%%%%%%%%%%%%%%%%%%%%%%%%%%%%%%%%%%%%%%%%%%%%%%%%%%%%%%%%%%%%%%
\section{Proofs of main results}\label{proofs}
%%%%%%%%%%%%%%%%%%%%%%%%%%%%%%%%%%%%%%%%%%%%%%%%%%%%%%%%%%%%%%%%%%%%%%%%%%%%%%%%%%%%%%%%%

%%%%%%%%%%%%%%%%%%%%%%%%%%%%%%%%%%%%%%%%%%%%%%%%%%%%%%%%%%%%%%%%%%%%%%%%%%%%%%%%%%%%%%%%%%%%%%%%%%%%%%%%%%%%%
\subsection{Proof of Theorem A}\label{proofA}
%%%%%%%%%%%%%%%%%%%%%%%%%%%%%%%%%%%%%%%%%%%%%%%%%%%%%%%%%%%%%%%%%%%%%%%%%%%%%%%%%%%%%%%%%%%%%%%%%%%%%%%%%%%%%%%%

Let $(\X, \rho)$ and $(\X', \rho')$ be two  separable metric spaces. The cartesian product space $\X\times \X'$ is defined with the metric $ \rho\times\rho'$ given by
\begin{equation*}
\rho\times\rho'\Big( (x_1, x'_1), (x_2, x'_2)\Big)= \max\Big\{\rho(x_1,x_2), \rho'(x'_1,x'_2)\Big).
\end{equation*}
Let $\xi  \in {\Phi}(\X)$ and $\xi'\in \Phi(\X')$. Let's define the function $\xi_0$ on $\mathcal{B}(\X)\times\mathcal{B}(\X')$ by :
$$
\xi_0(B\times B') = \xi(B) \; \xi'(B').
$$
%\subsection{Proof of Theorem A}
 The following lemma will be useful to study the zero-infinite case.
\begin{lemma}\label{zero-infinite}
Let $q \in \R$, $\delta >0,$ $E\subseteq \X$ and $E'\subseteq \X'$. Assume that $(\X, \xi, \mu)$ and $(\X', \xi', \nu)$ satisfy \eqref{hyp1}.
\begin{enumerate}
\item   If   $ {\W}_{\mu, \delta}^{q, \xi}(E) =\infty$ and $ {\W}_{\nu, \delta}^{q, \xi'}(E')= 0$ then ${\W}_{\mu\times \nu, \delta}^{q, \xi_0}(E\times E') =0.$
\item If   $ {\HHH}_{\mu, \delta}^{q, \xi}(E) =\infty$ and $ {\HHH}_{\nu, \delta}^{q, \xi'}(E')= 0$ then ${\HHH}_{\mu\times \nu, \delta}^{q, \xi_0}(E\times E') =0.$
\end{enumerate}
\end{lemma}

\begin{proof} We will only prove the first assertion. The other is a direct consequence, by taking each of the weights $c_i$ and $c'_{i_j}$ below to  be unity.
Let's consider $(c_i,B_i)_{i\geq 1}$ a weighted and centered $\delta$-cover of $E$ such that $(c_i \mu(B_i)^q \xi(B_i))$ is finite for each $i\geq 1$. Then, for $\epsilon>0$, we may choose for each $i$, a weighted  and centered $\delta$-cover $(c'_{i_j},B'_{i_j})_{j\geq 1}$ of $E'$ such that
\begin{equation}\label{zero_infty}
c_i\left(\mu(B_i)\right)^q\xi(B_i)\displaystyle\sum_{j=1}^\infty c'_{i_j}\left(\nu(B'_{i_j})\right)^q\xi'(B'_{i_j})<\frac{\epsilon}{2^i}
\end{equation}

For each $(x,x')\in E\times E'$, we have
\begin{align*}
\displaystyle\sum_{\stackrel{i,j=1}{ (x,x') \in B_i\times B'_{i_j}}}^\infty (c_ic'_{i_j})&=\sum_{\stackrel{i=1}{ x \in B_i}}^\infty\big(c_i\sum_{\stackrel{j=1}{ x' \in  B'_{i_j}}}^\infty c'_{i_j}\big)\\
&\geq\sum_{\stackrel{i=1}{ x \in B_i}}^\infty c_i \geq 1\,.
\end{align*}

Consequently, $(c_ic'_{i_j},B_i\times B'_{i_j})_{i,j\geq 1}$ is a weighted  and centered $\delta$-cover of $E\times E'$. Using (\ref{zero_infty}), we get :
\begin{align*}
{\W}_{\mu\times\nu, \delta}^{q,\xi_0}(E\times E')&\leq\displaystyle\sum_{i,j=1}^\infty c_i c'_{i_j}\mu\times\nu(B_i\times B'_{i_j})^q\xi_0(B_i\times B'_{i_j})\\
&=\sum_{i,j=1}^\infty c_ic'_{i_j}\mu(B_i)^q\nu(B'_{i_j})^q\xi(B_i)\xi'(B'_{i_j})\\
&=\sum_{i=1}^\infty c_i\mu(B_i)^q\xi(B_i)\Big(\sum_{j=1}^\infty c'_{i_j}\nu(B'_{i_j})^q\xi'(B'_{i_j})\Big)\\
&< \sum_{i=1}^\infty \frac{\epsilon}{2^i}= \epsilon\,.
\end{align*}
\end{proof}

\begin{proposition}\label{prop_equality}
Let $q\in \R$ and assume that $(\X, \xi,\mu)$ and $(\X', \xi', \nu)$ satisfy \eqref{hyp1}, then for any $E\subseteq \X$ and $E'\subseteq \X'$, we have :
\begin{equation}\label{equality_W}
{\W}_{\mu\times \nu, \delta}^{q, \xi_0}(E\times E') = {\W}_{\mu, \delta}^{q, \xi}(E)  {\W}_{\nu, \delta}^{q, \xi'}(E')\,,\quad\ \forall\ \delta>0\,.
\end{equation}
\end{proposition}

\begin{proof}
By Lemma \ref{zero-infinite}, we may suppose that  ${\W}_{\mu, \delta}^{q, \xi}(E) $ and $ {\W}_{\nu, \delta}^{q, \xi'}(E')$ are not one zero and the other infinite. We start by proving that:
\begin{equation}\label{inf_delta}
{\W}_{\mu\times \nu, \delta}^{q, \xi_0}(E\times E') \ge  {\W}_{\mu, \delta}^{q, \xi}(E)  {\W}_{\nu, \delta}^{q, \xi'}(E')\,.
\end{equation}
%Let $E_1\subseteq E$,  $E'_1\subseteq E'$ and $\delta >0$. First we will prove
%\begin{equation}\label{inf_delta}
%{\W}_{\mu\times \nu, \delta}^{q, \\\xi_0}(E_1\times E_1') \ge  {\W}_{\mu, \delta}^{q, \\xi}(E_1)  {\W}_{\nu, \delta}^{q, \\xi'}(E_1')
%\end{equation}
Without loss of generality, we may suppose that ${\W}_{\mu\times \nu, \delta}^{q, \xi_0}(E\times E')$ is finite and   $ {\W}_{\nu, \delta}^{q, \xi'}(E')$ is positive. Let $0< p<  {\W}_{\nu, \delta}^{q, \xi'}(E')$ and let $(c_i, B_i\times B_i')$ be a weighted and centered  $\delta$-cover of $E\times E'$. It follows that, for any $x'\in E'$, we have
$$
\sum_{ \stackrel{i=1}{x'\in B'_i}}^{\infty} c_i = \sum_{\stackrel{i=1}{ (x,x') \in B_i\times B'_i}}^{\infty} c_i \ge 1.
$$
Therefore, $(c_i, B_i')_{i\ge 1}$ is a weighted and centered $\delta$-cover of $E'$. We set, for each $i$,  $$u_i = c_i\,\nu(B'_i)^q\,  \xi'(B_i') /p ,$$ then
$$
p < \sum_{i=1}^{\infty} c_i\nu(B')^q \xi'(B'_i)\quad \text{and}\quad  \sum_{\stackrel{i=1}{ x\in  B_i}}^{\infty} u_i >1.
$$
Since this  holds for each $x\in E$, we get that $(u_i, B_i)_{i\ge 1}$ is a weighted and centered $\delta$-cover of $E$ and :
$$
\sum_{i=1}^{\infty} c_i \mu\times \nu(B_i\times B'_i)^q \xi_0(B_i\times B'_i) = p \sum_{i=1}^{\infty} u_i  \mu(B_i)^q \xi(B_i)\ge p   {\W}^{q, \xi}_{\mu,\delta}(E)
$$
Since $(c_i, B_i\times B_i')$ and $p$ are  arbitrarily chosen, we obtain \eqref{inf_delta}. \\
%and then
%$$
%{\W}^{q, \\\xi_0}_{\mu\times \nu,0}(E\times E') \ge {\W}^{q, \\xi}_{\mu,0}(E) {\W}^{q, \\xi'}_{\nu,0}(E')
%$$

Now, we will prove that
\begin{equation}\label{sup_delta}
{\W}^{q, \xi_0}_{\mu\times \nu, \delta} (E\times E') \le  {\W}^{q, \xi}_{\mu, \delta}(E) {\W}_{\nu, \delta}^{q, \xi'}(E').
\end{equation}
%Let $E_1\subseteq E$,  $E'_1\subseteq E'$ and $\delta >0$. First we will prove
%\begin{equation}\label{sup_delta}
%{\W}^{q, \\\xi_0}_{\mu\times \nu} \delta(E_1\times E'_1) \le {\W}^{q, \\xi}_{\mu, \delta} (E_1)  {\W}^{q, \\xi'}_{\nu,\delta}(E'_1)
%\end{equation}
We may assume that ${\W}^{q, \xi}_{\mu,\delta}(E)  {\W}^{q, \xi'}_{\nu,\delta} (E')$ is finite. Let $(c_i, B_i)_{i\ge 1}$ and $(c'_i, B'_i)_{i\ge 1}$  be weighted and centered $\delta$-covers for $E$ and $E'$ respectively. Then for $(x, x') \in E\times E'$ we have
$$
\sum_{\stackrel{i,j\geq 1}{ (x,x') \in B_i\times B'_j}} c_ic'_j = \Big( \sum_{\stackrel{i=1}{ x\in B_i}}^{\infty} c_i \Big) \Big( \sum_{\stackrel{j=1}{ x'\in B'_j}}^{\infty} c'_j \Big)  \ge 1.
$$
so that $(c_i c'_j, B_i\times B'_j)_{i ,j \ge 1}$ is a weighted and centered $\delta$-cover for $E\times E'$. As a consequence, we have
\begin{eqnarray*}
{\W}^{q, \xi_0}_{\mu\times \nu, \delta} (E\times E') &\le& \sum_{i,j=1}^{\infty} c_i c'_j \mu\times\nu(B_i\times B'_j)^q \xi_0(B_i\times B'_j)\\
&=& \sum_{i,j=1}^{\infty} c_i c'_j \mu(B_i)^q \nu( B'_j)^q \xi(B_i) \xi'(B'_j)\\
&=&\Big(  \sum_{i=1}^{\infty}   c_i  \mu(B_i)^q   \xi(B_i)   \Big) \Big( \sum_{j=1}^{\infty}   c'_j \nu(B_i)^q  \xi'(B_i')  \Big).
\end{eqnarray*}
Then \eqref{sup_delta} holds and this ends the proof.
\end{proof}

Now, we are able to give the proof of Theorem A. Let $E_1\subseteq E$ and $E'_1\subseteq E'$. By letting  $\delta \to 0$ in (\ref{equality_W}), we get :
\begin{align*}
{\W}_{\mu\times\nu,0}^{q,\xi_0}(E_1\times E'_1)&={\W}_{\mu,0}^{q, \xi}(E_1)  {\W}_{\nu,0}^{q, \xi'}(E'_1) \leq{\W}_{\mu}^{q, \xi}(E)  {\W}_{\nu}^{q, \xi'}(E')\,.
\end{align*}
  Therefore
 \[{\W}_{\mu\times\nu}^{q,\xi_0}(E\times E')\leq {\W}_\mu^{q,\xi}(E){\W}_\nu^{q,\xi'}(E')\,.\]
Moreover,
 \begin{align*}
{\W}_{\mu\times\nu}^{q,\xi_0}(E\times E')&\geq {\W}_{\mu\times\nu,0}^{q,\xi_0}(E_1\times E'_1)={\W}_{\mu,0}^{q,\xi}(E_1){\W}_{\nu,0}^{q,\xi'}(E'_1)\,
\end{align*}
and then, by arbitrariness of $E_1$ and $E'_1$,
 \begin{equation}\label{inequality_W}
 {\W}_{\mu\times\nu}^{q,\xi_0}(E\times E')\geq {\W}_{\mu}^{q,\xi}(E){\W}_{\nu}^{q,\xi'}(E')\,
 \end{equation}
as required.

\begin{remark} \label{rem}
\begin{enumerate}
\item  Under our convention $\infty \times 0= 0\times \infty=0$, the inequality
 \eqref{inequality_W} is true without the assumption \eqref{hyp1}.
\item Let $(\X,\xi)$ and $(\X', \xi')$ be two separable metric  spaces that satisfy :

 Any set $E$ of $\X$ and any set $E'$ of $\X'$ such that ${\HHH}_{\mu}^{q, \xi}(E)=0$ and ${\HHH}_{\mu}^{q, \xi'}(E')=0$,
  can be covered by countable numbers of balls $(B_i)_i$ and $(B'_j)_j$ of arbitrarily small diameters such that
 $\mu(B_i)^q\xi(B_i)=\mu(B'_j)^q\xi'(B'_j)=0$ for each $i$ and $j$.\\
\\
\noindent Then  the previous results remain true. Indeed, we only have to verify Lemma \ref{zero-infinite} under the previous hypthesis  instead of \eqref{hyp1}. In this case, we can choose a weighted and centered $\delta$-cover $(c_i, B_i)_i$ and $(c_j', B'_j)_j$ of $E$ and $E'$ such that, for each $j$, we have $\nu(B'_j)^q \xi'(B'_j)=0$. As in the proof of Lemma \ref{zero-infinite}, $(c_ic'_j, B_i\times B'_j)$ is a weighted and centered $\delta$-cover of $E\times E'$. thereby,
$$
{\W}_{\mu\times \nu}^{q, \xi}(E\times E') \le \sum_{i,j=1}^{\infty} c_ic'_j \mu(B_i)^q \xi(B_i)\nu(B'_j)^q \xi'(B'_j)=0
$$
which yield the desired result.
\end{enumerate}
\end{remark}

%%%%%%%%%%%%%%%%%%%%%%%%%%%%%%%%%%%%%%%%%%%%%%%%%%%%%%%%%%%%%%%%%%%%%%%%%%
\subsection{Proof of Theorem B}\label{application}
%%%%%%%%%%%%%%%%%%%%%%%%%%%%%%%%%%%%%%%%%%%%%%%%%%%%%%%%%%%%%%%%%%%%%%%%%%%%%%%%%%%%%%%%%%%%%%%%%%%%%%%%%%%%%%%%%%%%%%%%

A set  $E$ is  said to be regular if ${\W}_\mu^{q, \xi}(E) =  {\HHH}_\mu^{q, \xi}(E)$. Therefore, it is interesting to give some characterizations of such sets.
%In this section,  we let $\X$ to  be  separable metric space   such that  the Besicovitch covering theorem holds. !!!!!!!!!!!!! the Euclidean space $\R^d$. Therefore,  any regular and finite  Borel measure $\mu$ on $\R^d$ satisfies, for all $E\subseteq \R^d$,
%\begin{eqnarray*}
%\mu(E)&=&\sup\Big\{\mu(F),  \; F\subseteq E,\;\; E \; \text{closed}\; \Big\}\\
%&=&\inf\Big\{\mu(V),  \; E\subseteq V,\;\; V \; \text{open}\; \Big\},
%\end{eqnarray*}
%the reader may be referred  to \cite{Lorenz} for more details.
 In this section, we will be concerned to give a sufficient condition to get the regularity of $E.$
 %on the regularity equality on $\HHH_\mu^{q,\xi}$ and $\W_\mu^{q, \xi}$ (Theorem \ref{W=H}).  %First, let us observe  this  particular  case.
 %  when    $q\ge 1$ and  $t\ge 0$.
  %Now, in order to prove our main result in this section,
 We start by an  auxiliary result which is interesting in itself. We define, for $\epsilon >0$, the set  $$
\Lambda_{\mu, \epsilon}^{q,\xi}=\Big\{ x\in \X,\; \exists \delta_x>0,\; \HHH_\mu^{q,\xi}\big(B(x, r)\big) \le (1+\epsilon) \mu(B(x,r))^q \xi(B(x,r)),\; \forall r\le \delta_x\Big\}.
$$

\begin{lemma} \label{2.14}
Let $q\in \R, \xi\in \Phi_D(\X)$, $\epsilon >0$ and $\mu\in {\mathcal P}_D(\X)$. Assume that $\HHH_\mu^{q,\xi}(\X)<\infty$ then
$$
\HHH_\mu^{q,\xi}(\X) = \HHH_\mu^{q,\xi}(\Lambda_{\mu, \epsilon}^{q,\xi}).
$$
\end{lemma}
\begin{proof}
Let $E$ be the set of all points $x\in \X$ that are not in $\Lambda_{\mu, \epsilon}^{q,\xi}$, which means , every $x\in\X$  such that there exists $0<r_j\le \frac{1}{j}$ and
$$
\HHH_\mu^{q,\xi}\big(B(x,r_j)\big) >(1+\epsilon) \mu\big(B(x,r_j)\big)^q \xi(B(x, r_j))\,,\ \forall\ j\in\N.
$$
Suppose that there exists $F\subseteq E$ such that  $\HHH_\mu^{q,\xi}(F)>0$. Then, there exists an open set $U$ that satisfies $F\subseteq U$ and $\HHH_\mu^{q, \xi}(U)< \HHH_\mu^{q,\xi}(F)(1+\epsilon/4)$. Now, for $\delta>0$,  let's consider
$$
{\mathcal C} =\Big\{ B(x,r_j)\subseteq U; \;\; j\ge 5/\delta,\; x\in F\Big\}.
$$
Then $\mathcal{C}$ is a fine cover of $F$ (see definition \ref{def_fine_cover}). Moreover, by Theorem  \ref{Vit_2}, there exists a family of pairwise  disjoint balls  ${\mathcal C} '\subseteq {\mathcal C} $ and
$$
F\subseteq \bigcup_{B'\in {\mathcal C}'} 5 B'.
$$
It follows, since $\HHH_\mu^{q,\xi}(\X)<\infty$, that
%only countable many of them may have positive measure and the sum of measures  is finite.
 there exists a finite collection of balls   $B_1\ldots B_N\in {\mathcal C}'$ verifying
$$
\sum_{B'\in {\mathcal C}'\backslash\{B_1,\ldots, B_N\}} \HHH_\mu^{q,\xi}(B') <  C^{-1}\HHH_\mu^{q,\xi}(F) \frac\epsilon4,
$$
where $C$ satisfies  $\mu(B(x, 5r))  \xi(B(x, 5r))\le C \mu(B(x,r)) \xi(B(x,r))$, \; for all $x\in F$. Now, using again  Theorem \ref{Vit_2}, we obtain
\begin{equation}\label{vit_N}
F\subseteq \bigcup_{j=1}^N  B(x,r_j) \;\; \cup\;  \bigcup_{B'\in {\mathcal C}'\backslash\{B_1,\ldots,  B_N\}} 5B'.
\end{equation}

Since for each $B'\in {\mathcal C}'$, $B'\subset U$, $\diam(5B')\le 5 \diam(B')\le 2\delta$, we have
\begin{eqnarray*}
\HHH_{\mu,\delta}^{q,\xi}(F)&\le&\sum_{j=1}^N \mu(B(x,r_j))^q \xi(B(x,r_j)) + \sum_{B'\in {\mathcal C}'\backslash\{B_1,\ldots,  B_N\}} \mu(5 B')^q \xi( 5 B')\\
&\le& \frac1{1+\epsilon} \Big(\sum_{j=1}^N \HHH_\mu^{q,\xi}(B(x,r_j))+ C \sum_{B'\in {\mathcal C}'\backslash\{B_1,\ldots,  B_N\}} \HHH_\mu^{q,\xi}(B')\Big)\\
&\le& \frac1{1+\epsilon} \Big(\HHH_\mu^{q,\xi}(U) + \HHH_\mu^{q,\xi}(E)\frac\epsilon4\Big)\\
&\le& \HHH_\mu^{q,\xi}(E)\frac{1+\epsilon/2}{1+\epsilon}.
\end{eqnarray*}
By letting $\delta\to 0$ we obtain
$$
\HHH_{\mu, 0}^{q,\xi}(F)\le  \HHH_\mu^{q,\xi}(E)\frac{1+\epsilon/2}{1+\epsilon}<  \HHH_\mu^{q,\xi}(E).
$$
which is a contradiction since $F$ is an arbitrary subset of $E$.

\end{proof}
\begin{remark}
The result in Lemma \ref{2.14} remains true if  $q\le 0$ even without the assumption  $\mu\in {\mathcal P}_D(\X)$.
\end{remark}
Now, we are able to give the proof of Theorem B. For this, we may
assume by  Lemma \ref{W=H=infty} that $\HHH_{\mu}^{q,
\xi}(E)<\infty$. Let $\wt E$ be a Borel set such that $E\subseteq
\wt E$ and $\HHH_{\mu}^{q, \xi}(E)=\HHH_{\mu}^{q, \xi}(\wt E)$. Let
$\epsilon
>0$. For each $j\in \N$, we set
$$
W_j := \Big\{x\in \wt E,\;   \HHH_\mu^{q, \xi}(\wt E\cap B) \le
(1+\epsilon) \mu(B)^q \xi(B), \ \forall\ B\subset B(x,\frac{1}{j})\Big\}.
$$
Then, $\{W_j\}_j$ is an increasing set.
Fix $j\in\N$. Let $(c_k, B(x_k, r_k))_{k\ge 1}$ be a weighted and centered $(1/j)-$covering of $E$ such that
\begin{equation}\label{W_j}
\sum_{k=1}^{\infty} c_k \,  \mu(B(x_k, r_k))^q \xi(B(x_k,r_k)) \le \W_{\mu, 1/j}^{q,\xi}(E) +\epsilon.
\end{equation}
%Let $I= \{k : W_j \cap B(x_k,r_k) \neq \emptyset\}$, then
Then
$$E\cap W_j \subseteq \bigcup_{k=1}^{\infty} B(x_k,r_k)\cap W_j \quad \text{and}\quad \chi_{E\cap W_j}\leq \displaystyle\sum_{W_j \cap B(x_k,r_k) \neq \emptyset}c_k \chi_{\wt E\cap B(x_k,r_k)}\,.$$
Let $\mathcal{S} =\Big\{x\ |\ \displaystyle\sum_{W_j \cap B(x_k,r_k) \neq \emptyset}c_k\chi_{\wt E\cap B(x_k,r_k)}(x)\geq 1\Big\}\,.$
Then, $\mathcal{S}$ is a Borel set that contains $E\cap W_j$, and $\chi_{\mathcal{S}}\leq\displaystyle\sum_{W_j \cap B(x_k,r_k) \neq \emptyset} c_k\chi_{\wt E\cap B(x_k,r_k)}\,.$
Integrating with respect to $\HHH_\mu^{q,\xi}$ yields

\[\HHH_\mu^{q,\xi}(E\cap W_j)\le \HHH_\mu^{q,\xi}(\mathcal{S})\le\displaystyle\sum_{W_j \cap B(x_k,r_k) \neq \emptyset}c_k\HHH_\mu^{q,\xi}(\wt E\cap B(x_k,r_k))\,.\]
For $k\ge 1$ such that $W_j \cap B(x_k,r_k) \neq \emptyset$, we have
\[\HHH_\mu^{q,\xi}(\wt E\cap B(x_k,r_k))\le (1+\epsilon)\mu(B(x_k,r_k))^q \xi(B(x_k,r_k))\,.\]
Therefore,
\begin{align*}
\HHH_\mu^{q,\xi}(E\cap W_j)&\le (1+\epsilon)\displaystyle\sum_{W_j \cap B(x_k,r_k) \neq \emptyset}c_k \mu(B(x_k,r_k))^q\xi(B(x_k,r_k))\\
&\le (1+\epsilon)(\W_{\mu, 1/j}^{q,\xi}(E)+\epsilon)\,.
\end{align*}
When $\epsilon$ tends to $0$, we obtain $\HHH_\mu^{q,\xi}(E\cap W_j)\le  \W_{\mu, 1/j}^{q,\xi}(E)\,.$

Using Lemma \ref{2.14}, we have
$$
 \HHH_\mu^{q,\xi}\Big(\wt E \setminus \bigcup_{j=1}^{\infty} W_j   \Big)=0.
$$
Moreover, $\HHH_\mu^{q, \xi}$ is a regular measure. Then,
$$
\HHH_\mu^{q,\xi}(E) \le  \HHH_\mu^{q,\xi}\Big( E \cap  \bigcup_{j=1}^{\infty} W_j   \Big)= \lim_{j\to \infty}\HHH_\mu^{q,\xi}(E \cap W_j).
$$
Thus, $\HHH_\mu^{q,\xi}(E)\ \le \displaystyle\lim_{j\to+\infty} \W_{\mu, 1/j}^{q,\xi}(E) = \W_{\mu,0}^{q,\xi}(E)\le \W_\mu^{q,\xi}(E)\,.$

%\end{Proof of Theorem B}

\begin{remark}
The result in Theorem B remains true if  $q\le 0$ even without the assumption  $\mu\in {\mathcal P}_D(\X)$.
\end{remark}

%\newpage
%%%%%%%%%%%%%%%%%%%%%%%%%%%%%%%%%%%%%%%%%%%%%%%%%%%%%%%%%%%%%%%%%%%%%%%%%%%%%%%%
\subsection{Proof of Theorem C}\label{proofC}

%%%%%%%%%%%%%%%%%%%%%%%%%%%%%%%%%%%%%%%%%%%%%%%%%%%%%%%%%%%%%%%%%%%%%%%%%%%%%%%%%%%%%%%%%%%%%%%%%%%%%%%

%\begin{proposition}\label{rela_W_H}
Let  $q\in \R$,  $\mu\in {\mathcal P}_D(\X)$,  $\nu\in {\mathcal
P}_D(\X')$, $\xi \in \Phi(\X)$ and $\xi' \in \Phi(\X')$. Assume that
$(\X, \xi, \mu)$ and $(\X', \xi',\nu)$ satisfy \eqref{hyp1}. Let
$E\subseteq \X$, $E'\subseteq \X'$ and $\delta>0$. It is clear that
we only need to prove
\begin{equation}\label{W<H<W}
 {\W}_{\mu, \delta}^{q, \xi}(E)  {\HHH}_{\nu, \delta}^{q, \xi'}(E')
 \le {\HHH}_{\mu\times \nu, \delta}^{q, \xi_0}(E\times E') \le  {\HHH}_{\mu, \delta}^{q, \xi}(E)  {\HHH}_{\nu, \delta}^{q,
 \xi'}(E').
\end{equation}
The right hand side of \eqref{W<H<W} may be inferred from the proof
of (\ref{sup_delta}) by  taking $c_i$ and $c'_i$ to be unity.
Therefore, we will only prove the left hand side.  Without loss of
generality, we can suppose that $ {\HHH}_{\mu\times \nu, \delta}^{q,
\xi_0}(E\times E')$ is finite and $ {\HHH}_{\nu, \delta}^{q,
\xi'}(E')$ is positive. Let $p$ any number such that $0<p<
{\HHH}_{\nu, \delta}^{q, \xi'}(E')$ and $(B_i\times B'_i)_{i\geq 1}$
be any centered $\delta$-cover of $E\times E'$. It follows that
$(B'_i)_{i\ge 1}$ is a centered $\delta$-cover for $E'$. For each
$i$, let $u_i=\nu(B'_i)^q\xi'(B'_i)/p$, then
\[p<\displaystyle\sum_{i=1}^{\infty}\nu(B'_i)^q\xi'(B'_i)\quad\text{and}\quad
\sum_{\stackrel{i=1}{ x \in B_i}}^\infty u_i>1\,.\]

Since this holds for each $x\in E$, we obtain a weighted and centered $\delta$-cover $(u_i,B_i)_{i\geq 1}$ of $E$ and we have :
\begin{align*}
\displaystyle\sum_{i=1}^\infty\mu\times\nu(B_i\times B'_i)^q\xi_0(B_i\times B'_i)&=\sum_{i=1}^\infty \mu(B_i)^q\xi(B_i)\nu(B'_i)^q\xi'(B'_i)\\
&=p\sum_{i=1}^\infty u_i\mu(B_i)^q\xi(B_i)\geq p{\W}_{\mu,\delta}^{q,\xi}(E)\,.
\end{align*}

Since $(B_i\times B'_i)_{i\geq 1}$ and $p$ are arbitrarily chosen, we obtain \[{\W}_{\mu, \delta}^{q, \xi}(E)  {\HHH}_{\nu, \delta}^{q, \xi'}(E')
 \le {\HHH}_{\mu\times \nu, \delta}^{q, \xi_0}(E\times E')\,.\]

%%%%%%%%%%%%%%%%%%%%%%%%%%%%%%%%%%%%%%%%%%%%%%%%%%%%%%%%%%%%%%%%%%%%%%%%%%%%%%%%%%%%%%%%%%%%%%%%%%%%%%%%%%%%%%%%%%%%%%%%

\section{   Applications  }  \label{application_h}

%%%%%%%%%%%%%%%%%%%%%%%%%%%%%%%%%%%%%%%%%%%%%%%%%%%%%%%%%%%%%%%%%%%%%%%%%%%%%%%%%%%%%%%%%%%%%%%%%%%%%%%%%%%%%%%%%%%%%
%%%%%%%%%%%%%%%%%%%%%%%%%%%%%%%%%%%%%%%%%%%%%%%%%%%%%%%%%%%%%%%%%%%%%%%%%%%%%%%%%%%%%%%%%%%%%%%%%%%%%%%%%%%%%%%%%%%%%

\subsection{ Application 1 : Definition of $\xi$ using Hausdorff function }\label{Haus}
%%%%%%%%%%%%%%%%%%%%%%%%%%%%%%%%%%%%%%%%%%%%%%%%%%%%%%%%%%%%%%%%%%%%%%%%%%%%%%%%%%%%%%%%%%%%%%%%%%%%%%%%%%%%%%%%%%%%%
For any Hausdorff function $h$, we set $\xi_h\in\Phi(\X)$ defined on $\mathcal{B}(\X)$ by $\xi_h(B) = h(\diam_{\rho}(B))$ where  $\diam_{\rho}(B)$ denotes the diameter of $B$ with respect to $\rho$.
Let $h, h'\in {\F}$. We define   the  function $\xi_0$, on the set ${ \mathcal B}(\X)\times { \mathcal B}(\X')$,  by \\
\[\xi_0(B\times B')=\xi_h(B)\xi_{h'}(B')=\left\{\begin{array}{lcl}
0&&\text{if}\ B\times B'=\emptyset\\
&&\\
h(\diam_{\rho}(B))h'(\diam_{\rho'}(B'))&& \text{if not}.
\end{array}\right.\]\\
We recall that the Hausdorff function $h\times h'$ is defined by \[h\times h'(r)=h(r)h'(r),\quad \quad \text{for all}\ r\in\R^+\,.\]

\begin{theorem}\label{thm_W}
For any $E\subseteq\X$ and any $F\subseteq\X'$, we have :
\[\W_{\mu\times\nu}^{q,{h\times h'}}(E\times F)\geq \W_\mu^{q,h}(E)\W_\nu^{q,{h'}}(F)\,.\]
Assume that either  $\W_\mu^{q,h}(E) =  \HHH_\mu^{q,h}(E)$ or  $\W_\nu^{q,h'}(F)=  \HHH_\nu^{q,h'}(F)$, then
\begin{equation}\label{classic_HH}
\HHH_{\mu\times\nu}^{q,h\times h'}(E\times F)\geq \HHH_\mu^{q,h}(E)\HHH_\nu^{q,h'}(F)\,.\end{equation}
\end{theorem}

\begin{proof}
First, we will prove a simple relation  between the generalized measures generated by $\xi_0$ and those generated by $\xi_{h\times h'}$. Let $\delta>0$,   $E_1\subseteq E$ and  $F_1\subseteq F$. We will prove that
\begin{equation}\label{W-h-phi}
\W_{\mu\times\nu, \delta }^{q, {h\times h'}}(E_1\times F_1)\geq \W_{\mu\times\nu, \delta }^{q, \xi_0}(E_1\times F_1)
\end{equation}
Clearly, we may  assume that $\W_{\mu\times\nu,\delta}^{q,\xi_{h\times h'}}(E_1 \times F_1)$ is finite. We consider a weighted and centered $\delta$-cover $(c_i,B_i\times B'_i)_{i=1}^\infty$ for $E_1\times F_1$ and, for each $i$, we set
 \[d^{\times}(B_i\times B'_i):=\max (\diam_{\rho}(B_i),\diam_{\rho'}(B'_i))\,.\]
Then, \begin{align*}
\xi_{h\times h'}(B_i\times B'_i) = h\times h'(d^{\times}(B_i\times B'_i))&=h(d^{\times}(B_i\times B'_i))\, h'(d^{\times}(B_i\times B'_i))\\
&\geq h(\diam_{\rho}(B_i))\, h'(\diam_{\rho'}(B'_i))\\
&=\xi_0(B_i\times B'_i).
\end{align*}

After summation, we get

\begin{eqnarray*}\label{ineq_sum}
\displaystyle\sum_{i=1}^\infty c_i\mu\times\nu(B_i\times B'_i)^q h\times h'(d^{\times}(B_i\times B'_i))&\geq&\sum_{i=1}^\infty c_i\mu\times\nu(B_i\times B'_i)^q\xi_0(B_i\times B'_i)\\
&\geq& \W_{\mu\times\nu,\delta}^{q,\xi_0}(E_1\times F_1).
\end{eqnarray*}
Which gives \eqref{W-h-phi} and implies that
$\W_{\mu\times\nu,0}^{q,{h\times h'}}(E_1\times F_1)\geq
\W_{\mu\times\nu,0}^{q,\xi_0}(E_1\times F_1)$ for any subsets
$E_1\subseteq E$ and $ F_1\subseteq F. $ Therefore,
$$\W_{\mu\times\nu}^{q,{h\times h'}}(E\times F)\geq \W_{\mu\times\nu}^{q,\xi_0}(E\times F).$$
Now, using  \eqref{eq_product_W}, we obtain
\begin{eqnarray*}
\W_{\mu\times\nu}^{q,{h\times h'}}(E\times F)&\geq&
\W_{\mu}^{q,h}(E)\W_{\nu}^{q,{h'}}(F)
\end{eqnarray*}

Similarly, by taking  $c_i$ to be equal to $1$,  we obtain
$ \HHH_{\mu\times \nu}^{q,{h\times h'}}(E\times F)\geq \HHH_{\mu\times\nu}^{q,\xi_0}(E\times F)\,.$
Moreover, using Theorem C and  \eqref{eq_thB}, we get
\begin{equation}\label{H-W}
\HHH_{\mu\times\nu}^{q, {h\times h'}}(E\times F)\geq \begin{cases}
 \W_\mu^{q,h}(E)\HHH_\nu^{q, {h'}}(F)&\\
 & \\
 \HHH_\mu^{q,h}(E)\W_\nu^{q, {h'}}(F).&\\
 \end{cases}
 \end{equation}
So, if we suppose that $\W_\mu^{q,h}(E)=\HHH_\mu^{q,h}(E)$ or  $\W_\mu^{q,h'}(F)=  \HHH_\mu^{q,h'}(F)$  we obtain
\[\HHH_{\mu\times\nu}^{q, {h\times h'}}(E\times F)\geq \HHH_\mu^{q, h}(E)\HHH_\nu^{q, {h'}}(F)\,.\]
\end{proof}

\begin{remark}
As a consequence of theorem B, if $h$ and $\mu$ (or $h'$ and $\nu$) are blanketed, then \eqref{classic_HH} holds.
\end{remark}

%%%%%%%%%%%%%%%%%%%%%%%%%%%%%%%%%%%%%%%%%%%%%%%%%%%%%%%%%%%%%%%%%%%%%%%%%%%%%%%%%%%%
\subsection{Application 2}\label{sec_density}
%%%%%%%%%%%%%%%%%%%%%%%%%%%%%%%%%%%%%%%%%%%%%%%%%%%%%%%%%%%%%%%%%%%%%%%%%%%%%%%%%%%%

%%%%%%%%%%%%%%%%%%%%%%%%%%%%%%%%%%%%%%%%%%%%%%%%%%%%%%%%%%%%%%%%%%%%%%%%%%%%%%%%%%%%%%%%%%%%%%%%%%%%%%%%%%%%%%%%%%%%%

%%%%%%%%%%%%%%%%%%%%%%%%%%%%%%%%%%%%%%%%%%%%%%%%%%%%%%%%%%%%%%%%%%%%%%%%%%%%%%%%%%%%
%It is well know that density theorems play a major role in geometric  measure theory. \\

 Here we will consider an interesting case when  $\xi_0$ is defined on
$\mathcal{B}(\X\times\X')$ by $\xi_0(B\times B') =\xi(B)\xi'(B')$.
Our goal is to   prove Theorem \ref{product_partial} using density
approach. Let $\nu,~\mu \in \mathcal{P}(\X)$,   $x \in supp(\mu)$,
$q \in \R$ and $h \in {\mathcal F}_D$, we define the upper
$(q,\xi)$-density at $x$ with respect to $\mu$ by
$$\displaystyle\overline{d}_{\mu}^{~q,\xi}(x,~\nu) = \underset{r \searrow 0}{\limsup}~\displaystyle\frac{\nu  \left( B(x,~r) \right)}{\mu\left( B(x,~r) \right)^{q} \xi\left( B(x,r) \right)}. $$

%In the following, we will give our density theorem.
\begin{lemma}\label{density}
Let $\mu, \nu\in {\mathcal P}(\X), q\in \R, \xi \in {\Phi}_D(\X)$ and $E\subseteq \supp \mu$.
\begin{itemize}
\item[(i)] If $\mu \in \mathcal{P}_D(\X)$ and $\HHH_{\mu}^{~q,\xi}(E) < \infty$ then
\begin{equation}\label{dens1}
\mathcal{H}_{\mu}^{~q,\xi}(E)~ \underset{x \in E}{\inf}~\displaystyle\overline{d}_{\mu}^{~q,\xi}(x,~\nu) \leq \nu(E)
\end{equation}
\item[(ii)] If $\mathcal{H}_{\mu}^{~q, \xi}(E) < \infty$ then
\begin{equation}\label{dens2}
\nu(E) \leq \mathcal{H}_{\mu}^{~q,\xi}(E)~ \underset{x \in E}{\sup}~\displaystyle\overline{d}_{\mu}^{~q, \xi}(x,~\nu).
\end{equation}
\end{itemize}
\end{lemma}
\begin{proof}
The proof of this lemma  is straightforward and mimics that in Theorem 2.14 in \cite{Ol1},   when we use Theorem \ref{Vit_2} instead of  \cite[Lemma 1.9]{Falconer}.
\end{proof}

An  interesting consequence of Lemma \ref{density} is the following
Corollary which will be used to prove  Theorem
\ref{product_partial}.
 \begin{corollary} \label{Hmu} Let $\xi\in {\Phi}_D(\X)$,  $q\in \R,$ $\mu\in {\mathcal P}_D(\X)$  and $E\subseteq \supp \mu $ be a Borel set  such
that ${\HHH}_\mu^{q, \xi}(E) <+\infty$.
\begin{enumerate}
\item  If   $\overline{d}_{\mu}^{q, \xi}(x,  {\W_\mu^{q,\xi}}_{\llcorner E}) =+\infty$ for ${{ \HHH}_\mu^{q, \xi}}$-a.a. on $E$, then  $${\W}_\mu^{q, \xi}(E) = {\HHH}_\mu^{q, \xi}(E)=0.$$
\item If  ${{\W}_\mu^{q, \xi}}(E)=0$ then
$$  \overline{d}_\mu^{q, \xi}(x, {\W_\mu^{q,\xi}}_{\llcorner E})=0 \;\; \text{for}\;\; {{ \HHH}_\mu^{q, \xi}}-a.a.\;\;  \text{on}\;\; E.$$
\item $ \overline{d}_{\mu}^{q,\xi}(x, {{\HHH}_\mu^{q, \xi}}_{\llcorner E}) =1, $ $ \HHH^{q,\xi}_\mu\text{-a.a. on } \; E.$
\item If  there exists $\nu\in  \mathcal {P}(\X)$ such that
$\sup_{x\in E} \overline{d}^{q, \xi}_\mu(x, \nu)\le  \gamma < \infty$ then
 $$\HHH^{q, \xi}_\mu (E) \ge \nu(E)/\gamma.$$
\end{enumerate}

\end{corollary}

\begin{proof}
\begin{enumerate}
\item We take $\theta ={{\W}_\mu^{q, \xi}}_{\llcorner E}$ and we set  $F=\Big\{x\in E;\;\;\overline{d}_\mu^{q,\xi}(x, \theta)=+\infty\Big\}$ and we obtain from \eqref{dens1} that
$$
{\mathcal H}_\mu^{q,\xi}(F) \displaystyle\inf_{x\in F}
\overline{d}_\mu^{q,\xi}(x,\theta)\leq \theta (F)<+\infty.
$$
This implies that  ${\mathcal H}_\mu^{q,\xi}(F)=0$ and ${\mathcal
H}_\mu^{q,\xi}(E)={\mathcal H}_\mu^{q,\xi}(F)+{\mathcal H}_\mu^{q,\xi}(E\setminus F)=0$.
\item  We take $\theta ={{\W}_\mu^{q, \xi}}_{\llcorner E}$ and we consider, for $n\in \mathbb{N}$,  the set
 $$
E_n = \Big\{x\in E;\;\;\overline{d}_\mu^{q,\xi}(x, \theta)\ge 1/n\Big\}.
 $$
Then \eqref{dens1} gives
$$
 {{\mathcal H}_\mu^{q, \xi}}(E_n) \le n \,  \theta ( E_n)  =0,\;\;
\forall n\in\mathbb{N}.
$$
We therefore conclude that $\overline{d}_\mu^{q, \xi}(x,\theta)= 0 $ for $ {{\mathcal H}_\mu^{q, \xi}}$-a.a. on $E$.\\
\item Take $\theta ={{\HHH}_\mu^{q, \xi}}_{\llcorner E}$. Put the set $ F=\Big\{ x\in E;\;\; \overline{d}_\mu^{q, \xi}(x, \theta)>1\Big\},$ and for $m\in\mathbb{N}^*$
 $$
F_m=\left\{ x\in E;\;\; \overline{d}_\mu^{q, \xi}(x, \theta)>1+ \frac1m\right\}.
 $$
We therefore deduce from \eqref{dens1}  that
 $$
\left(1+ \frac1m\right){\HHH}_\mu^{q, \xi}(F_m)\leq {\HHH}_\mu^{q, \xi}(F_m).
 $$
This implies that ${\HHH}_\mu^{q, \xi}(F_m)=0$. Since  $F=\bigcup_m
F_m$, we obtain ${\HHH}_\mu^{q, \xi}(F)=0$, i.e.
\begin{eqnarray}\label{33bBBB}
\overline{d}_\mu^{q, \xi}(x, \theta)\leq1 \quad \text{for} \quad {\HHH}_\mu^{q, \xi}\text{-a.a.}\; x\in E.
\end{eqnarray}
Now consider the set $ \tilde F=\Big\{ x\in E;\;\;
\overline{d}_\mu^{q, \xi}(x, \theta)<1\Big\},$ and for $m\in\mathbb{N}^*$
 $$
\tilde F_m=\left\{ x\in E;\;\; \overline{d}_\mu^{q, \xi}(x, \theta)<1-
\frac1m\right\}.
 $$
Using \eqref{dens2}, we clearly have
 $$
{\HHH}_\mu^{q, \xi}(\tilde F_m)
\le \left(1-\frac1m\right){\HHH}_\mu^{q, \xi}(\tilde F_m).
 $$
This implies that ${\HHH}_\mu^{q, \xi}(\tilde F_m)=0$. Since
$F=\bigcup_m \tilde F_m$, we obtain ${\HHH}_\mu^{q, \xi}(F)=0$,
i.e.
\begin{eqnarray}\label{33bNN}
\overline{d}_\mu^{q, \xi}(x, \theta) \geq1 \quad \text{for} \quad {\HHH}_\mu^{q, \xi} \text{-a.a.}\; x\in E.
\end{eqnarray}
The result  follows from \eqref{33bBBB}  and \eqref{33bNN}.
\item A direct consequence of \eqref{dens2}.
\end{enumerate}

\end{proof}

\begin{theorem}\label{product_partial}
Let $\mu \in {\mathcal P}_D(\X)$,   $\nu \in {\mathcal P}_D(\X'),$
$\xi \in\Phi_D(\X)$, $\xi'\in \Phi(\X')$, $q\in \R,$ $E$ and $E'$ be two Borel  sets of
$\supp\mu$ and $\supp \nu$ respectively.  Assume that
$\HHH^{q,\xi}_{\mu}(E) <+\infty $ and $\HHH^{q, \xi'}_{\nu}(E')
<+\infty.$ then
$$
\HHH^{q,\xi}_{\mu}(E) \HHH^{q,\xi'}_{\nu}(E') \le   \HHH^{q, \xi_0}_{\mu \times \nu}(E \times E').
$$
\end{theorem}

\begin{proof}
 Let $\theta_1$ be  the restriction of $\HHH^{q,\xi}_{\mu}$ to $E$ and $ \theta_2$  be the restriction
of $\HHH^{q, \xi'}_{\nu}$ to $E'.$ We set
$$
\widetilde E =  \left\{  x\in E,\;\; \;\; {\overline
d}^{q,\xi}_{\mu}(x, \theta_1) =  1 \right\}
$$
and
$$
\widetilde E' =  \left\{  x\in E',\;\; \;\; {\overline
d}^{q,\xi'}_{\nu} (x, \theta_2) =  1  \right\}.
$$
Then, using corollary   \ref{Hmu}, we have $\theta_1(E) =
\theta_1(\widetilde E) $ and $\theta_2(E') = \theta_2(\widetilde E')$. Therefore,  for
$(x, y) \in \widetilde E \times \widetilde E$, we have
\begin{eqnarray*}
{\overline d}_{\mu\times \nu}^{q, \xi_0} \Big( (x, y), \theta_1
\times \theta_2\Big) &=& \limsup_{r\to 0} \left[\frac{\theta_1\big(B(x,r)
\big)}{\mu\big(B(x,r)  \big)^q \xi(B(x,r))}\;
\frac{\theta_2\big(B(y,r)\big)}{\nu\big(B(y,r)
  \big)^q \xi(B(x, r))} \right] \\
&\le & {\overline d}_{\mu}^{q, \xi} \big(x, \theta_1 \big) \;\;
{\overline d}_{\nu}^{q, \xi'} \big(y, \theta_2 \big) =  1.
\end{eqnarray*}
Therefore, by corollary  \ref{Hmu},
$$
 \HHH^{q, \xi_0}_{\mu\times \nu} \big( \widetilde E \times
\widetilde E'\big) \ge   \theta_1\times \theta_2  \big( \widetilde E \times
\widetilde E'\big) =  \theta_1 \big( \widetilde E\big) \theta_2\big(
\widetilde E'\big)  =  \theta_1(E) \theta_2(E').
$$
Hence,
$$
\HHH^{q, \xi_0}_{\mu\times \nu} (E\times F) \ge
\HHH^{t+s}_{\mu\times \nu} \big( \widetilde E \times \widetilde
F\big) \ge  \theta_1(E) \theta_2(E') =  \HHH^{q, \xi}_{\mu}(E) \HH^{q,
\xi'}_{\nu}(E').
$$
\end{proof}

%\newpage

%\newpage

\section{appendix}
\begin{definition}\label{def_fine_cover}\cite{Edgar98}\\
A fine cover of a set $E$ is a family $\mathcal{C}$ of closed balls $B(x,r)$ with $x\in E$ and $r>0$ such that, for every $x\in E$ and every $\epsilon>0$, there is $r>0$ such that $r<\epsilon$ and $B(x,r)\in\mathcal{C}$.
\end{definition}

\begin{theorem}[Besicovitch covering Theorem]\label{BCT}\cite{Mat}.\\
There exists an integer $\gamma \in\N$ such that, for any subset $A$ of
$\R^n$ and any sequence $(r_x)_{x\in A}$ satisfying
\begin{enumerate}
\item $r_x>0$, \quad $\forall\; x\in A$,
\item $\displaystyle\sup_{x\in A} r_x < \infty$.
\end{enumerate}

Then, there exists $\gamma$ countable or  finite families $B_1, \ldots,
B_\gamma$ of $\big\{ B_x(r_x), \;\; x\in A \big\}$, such that
\begin{enumerate}
\item $A\subset \bigcup_i \bigcup_{B\in B_i} B$.
\item $B_i$ is a family of disjoint sets.
\end{enumerate}
\end{theorem}
%\begin{lemma}[Vitali's lemma]\label{Vit_l}\cite{Mat}.\\
%Let $X$ be a boundedly compact metric space and $\mathcal{B}$ a
%family of closed balls in $X$, such that
% $$
%\sup\Big\{\dim (B),\;\; B\in \mathcal{B}\Big\}<\infty.
% $$

%Then, there exists a finite on countable sequence of disjoint balls
%$(B_i)_i\subset \mathcal{B}$, such that
% $$
%\displaystyle\bigcup_{B\in \mathcal{B}} B\subset \bigcup_i
%\big(5B_i\big).
% $$
%\end{lemma}

%\begin{theorem}[Vitali 1]\label{Vit_1}\cite{Mat}.

%Let $\mu$ be a Radon measure on $\R^n$, $A\subset\R^n$ and
%$\mathcal{B}$ a family of closed balls, such that each point of $A$
%is the center of an arbitrarily small ball of $\mathcal{B}$, i.e.
% $$
%\inf\Big\{r,\quad B_x(r)\in\mathcal{B}\Big\}=0, \qquad
%\text{for}\;\; x\in A.
% $$

%Then, there exists a family of disjoint balls $(B_i)_i\subset
%\mathcal{B}$, such that
% $$
%\mu\Big(A\setminus \bigcup_i B_i\Big)=0.
% $$
%\end{theorem}

\begin{theorem}[Vitali 2]\label{Vit_2}\cite{Edgar07}
Let $X$ be a metric space, $E$ a subset of $X$ and $\mathcal{B}$ a
family of fine cover of $E$. Then, there exists either

\begin{enumerate}
\item an infinite (centered closed ball) packing $\big\{ B_{x_i}
(r_i)\big\}_i\subset \mathcal{B}$, such that $\inf \{r_{i}\} > 0$,

or

\item a countable (possibly finite) centered closed ball packing $\big\{ B_{x_i}
(r_i)\big\}_i\subset \mathcal{B}$, such that for all $k\in\N$,
 $$
E\setminus\bigcup_{i=1}^k B_{x_i}(r_i)\subset\bigcup_{i\ge
k}B_{x_i}(5r_i).
 $$
\end{enumerate}
\end{theorem}

\end{document}